\documentclass[12pt, reqno]{amsart} 
\usepackage{amsmath, amssymb, amsfonts, amstext, amsthm, amscd}
\usepackage[mathscr]{euscript}
\usepackage{enumerate}
\usepackage{tikz,color,soul}
\usepackage[english]{babel}
\usepackage{chngcntr}
\usepackage{stmaryrd}
\usepackage{a4wide}

\usepackage[pagebackref]{hyperref}

\usepackage[left=2.5cm,top=2.5cm,bottom=3cm,right=2.5cm]{geometry}

\usetikzlibrary{arrows.meta, positioning,arrows}
\makeatletter
\newtheorem*{rep@theorem}{\rep@title}
\newcommand{\newreptheorem}[2]{%
	\newenvironment{rep#1}[1]{%
		\def\rep@title{#2 \ref{##1}} 
		\begin{rep@theorem}}%
		{\end{rep@theorem}}}
\makeatother

\newreptheorem{conjecture}{Conjecture}

\newtheorem{theorem}{Theorem}[section]

\theoremstyle{definition}
\newtheorem{definition}[theorem]{Definition}

\newtheorem{remark}[theorem]{Remark}

\theoremstyle{plain}
\numberwithin{equation}{section}
\newtheorem*{theorem*}{Theorem}
\newtheorem{proposition}[theorem]{Proposition}

\newtheorem{conjecture}[theorem]{Conjecture}

\theoremstyle{definition}

\numberwithin{equation}{section}

\title{Rationality of Seshadri constants on blow-ups of ruled surfaces}

 \author[K. Hanumanthu]{Krishna Hanumanthu}
\address{Chennai Mathematical Institute, H1 SIPCOT IT Park, Siruseri, Kelambakkam 603103, India}
 \email{krishna@cmi.ac.in}

 \author[C. J.  Jacob]{Cyril J. Jacob}
 \address{Chennai Mathematical Institute, H1 SIPCOT IT Park, Siruseri, Kelambakkam 603103, India}
 \email{cyril@cmi.ac.in}

\author[Suhas B. N.]{Suhas B. N.}
 \address{Chennai Mathematical Institute, H1 SIPCOT IT Park, Siruseri, Kelambakkam 603103, India}
 \email{suhasbn@cmi.ac.in}

 \author[A. K. Singh]{Amit Kumar Singh}
 \address{Chennai Mathematical Institute, H1 SIPCOT IT Park, Siruseri, Kelambakkam 603103, India}
 \email{amitksingh@cmi.ac.in}



\date{\today}
\begin{document}
\begin{abstract}
		In this note, we continue the study of Seshadri constants on blow-ups of Hirzebruch surfaces initiated in \cite{HJSS2024}. Now we consider blow-ups  of ruled surfaces more generally. We propose a conjecture for classifying all the negative self-intersection curves on the blow-up of a ruled surface at very general points, analogous to the $(-1)$-curves conjecture in $\mathbb{P}^2$. Assuming  this conjecture is true, we  exhibit an ample line bundle with an irrational Seshadri constant at a very general point on such a surface. 
\end{abstract}
\maketitle

\section{Introduction}\label{Introduction}

Seshadri constants measure local positivity of line bundles on projective varieties and a lot of research has focussed on questions related to Seshadri constants. One of the fundamental open questions about Seshadri constants is whether they can be irrational. Some recent work on this question can be found in \cite{Dumnicki2016, Hanumanthu2018,Farnik2020}. In these works, irrational Seshadri constants are exhibited on general blow-ups of the complex projective plane $\mathbb{P}^2$ assuming certain conjectures are true. 

In this note, we prove analogous results for general blow-ups of ruled surfaces, assuming certain conjectures are true. 
In the special case of Hirzebruch surfaces, results in this note give an answer to \cite[Question 4.11]{HJSS2024}. 

For a detailed literature and introduction to the problem, see Introduction of \cite{HJSS2024}. 

In Section \ref{section2}, we list some conjectures on ruled surfaces which are natural generalizations of well-known conjectures on $\mathbb{P}^2$ and Hirzebruch surfaces. Conjecture \ref{conjecture-1} and Conjecture \ref{conjecture-2} are new and are equivalent to each other.  In Proposition \ref{NBS analog in our case}  we relate these conjectures to the well-known 
Nagata-Biran-Szemberg Conjecture (Conjecture \ref{Nagata Biran Szemberg Conjecture}). 

In Section \ref{section3}, we study Seshadri constants. After proving some preliminary results about them, we prove our main result on irrational Seshadri constants which is stated below. Our argument is inspired by that of \cite{Hanumanthu2018}. To generalize that proof, we introduce a notion called \textit{good form} for divisors (Definition \ref{definition of good form}) on blow-ups of ruled surfaces which is analogous to standard form for divisors on blow-ups of $\mathbb{P}^2$. 

\begin{theorem*}[Theorem \ref{irrationality theorem}]
Let $X$ be a ruled surface and let $X_r$ be the blow-up of $X$ at $r$ very general points. 
Suppose that Conjecture \ref{conjecture-1} is true. Then there exists a pair $(r, L)$ where $r$ is a positive integer and $L$ is an ample line bundle on $X_r$ such that $\varepsilon(X_r,L,x)$ is irrational for a very general point $x \in X_r$. 

    In fact, such a pair exists for infinitely many $r$. 
\end{theorem*}

We work over the field of complex numbers. We follow the notation of \cite{HJSS2024}. For preliminaries and notation on ruled surfaces, refer to \cite[Section V.2]{Hartshorne1977}.

\section{Some  Conjectures}\label{section2}
 Throughout this note $X$ denotes a ruled surface over a complex non-singular projective curve $\Gamma$ with invariant $e$.  Denote the genus of $\Gamma$ by $g$. 
 
 We fix $r$ very general points $p_1,\dots,p_r \in X$ and consider the blow-up  $\pi: X_r \rightarrow X$ of $X$ at $p_1,\dots,p_r$. The normalized section  $\Gamma_e$ and the fiber $f$ form a  basis of the free abelian group $\text{Num}(X)$, the Picard group of $X$ modulo numerical equivalence. Their pullbacks in $X_r$ are denoted by $H_e$ and $F_e$, respectively. 
 By abuse of notation, we will use $H_e$ and $F_e$ to denote divisor classes as well as specific
curves linearly equivalent to these divisor classes. 
 
 The canonical divisor classes on $X$ and $X_r$ are denoted by $K_X$ and $K_{X_r}$ respectively. 
 If $D \subset X$ is a curve then
$ \widetilde{D}$ denotes the strict transform of $D$ on $X_r$. 

 We emphasize that by equality of divisors we always mean equality in $\text{Num}(X)$.

We now state two conjectures on ruled surfaces. The first one characterizes curves on $X_r$ which have negative self-intersection. The second gives a lower bound on the self-intersection of a reduced and irreducible curve in terms of its multiplicities at very general points. We later show that the two conjectures are equivalent. 

 \begin{conjecture}\label{conjecture-1}
Let $X$ be a ruled surface and let $X_r \to X$ be a blow up of $X$ at $r$ very general points.  Let $C$ be a reduced irreducible curve in $X_r$ with $C^2 <0$. Then one of the following holds:
\begin{enumerate}
\item[(i)] $C$ is a $(-1)$-curve,
\item[(ii)] $C = \widetilde{\mathcal{C}}$, where $\mathcal{C}$ is a curve on $X$ with $\mathcal{C}^2 \le 0$, or  
\item[(iii)] $C = \widetilde{\Gamma_e}$.
\end{enumerate}
\end{conjecture}
\begin{remark}
    When $X$ is a Hirzebruch surface, the above conjecture is equivalent to \cite[Conjecture 4.8]{HJSS2024}.
\end{remark}
\begin{conjecture}\label{conjecture-2}
Let $X$ be a ruled surface and let $p_1,\dots, p_r \in X$ be very general points. Let $\mathcal{C}$ be a reduced irreducible curve in $X$ 
and define  $m_i = \text{mult}_{p_i}\mathcal{C}$ for $1 \le i \le r$. Assume that $m_i > 0$ for some $i$. Then 
$\mathcal{C}^2 \ge \sum\limits_{i =1}^r {m}_i^2 -1$. Further, for equality, we have the following:
\begin{enumerate}
    \item[(i)] If $e \neq 0$, the equality holds only if $\widetilde{\mathcal{C}}$ is a $(-1)$-curve or $\widetilde{\mathcal{C}} =  \widetilde{\Gamma_e}$, where $\widetilde{\mathcal{C}}$ is the strict transform of $\mathcal{C}$ in $X_r$.
    \item[(ii)] If $e = 0$, the equality holds only if $\widetilde{\mathcal{C}}$ is a $(-1)$-curve or $\widetilde{\mathcal{C}} =  \widetilde{\alpha\Gamma_e}$, where $\alpha \geq 1$.
\end{enumerate}

\end{conjecture}

\begin{proposition}\label{equivalence of conjectures 1 and 2}
Conjectures \ref{conjecture-1} and \ref{conjecture-2} are equivalent.

\begin{proof}
Suppose that Conjecture \ref{conjecture-1} is true.  
Let $\mathcal{C}$ be as mentioned in Conjecture \ref{conjecture-2}. Suppose that
\begin{eqnarray}\label{eq:1}
\mathcal{C}^2 < \sum\limits_{i=1}^r m_i^2 -1.
\end{eqnarray} 
Then ${\widetilde{\mathcal{C}}}^2 = \mathcal{C}^2 - \sum\limits_{i=1}^r m_i^2 < -1$. This implies that $\widetilde{\mathcal{C}}$ is a negative self-intersection curve in $X_r$. So, by Conjecture \ref{conjecture-1}, we have 
\begin{enumerate}
\item[(i)] $ \mathcal{C}^2 \le 0$ or
\item[(ii)] $\widetilde{\mathcal{C}} = \widetilde{\Gamma_e}$.
\end{enumerate}

Now, as $\mathcal{C}$ passes through at least one $p_i$, by \cite{Ein-Lazarsfeld1993}, 
\begin{eqnarray}\label{eq:2}
\mathcal{C}^2 \ge 0.
\end{eqnarray}
We now consider both the cases separately.

\noindent
\textbf{Case (i):} $\mathcal{C}^2 \le 0$. 

This, along with \eqref{eq:2} implies that $\mathcal{C}^2 = 0$. So, again by \cite{Ein-Lazarsfeld1993}, $m_i \le 1$ for each $i$. We also have 
\begin{equation} \label{Xu's like inequality}
    \mathcal{C}^2 \ge \sum\limits_{i =1}^r m_i^2  -m,
\end{equation}
where $m = \min \{ m_i \; \vert \; m_i \neq 0 \}$ (see \cite{Ein-Lazarsfeld1993} and \cite[Lemma 1]{Xu1994}). Note that  $m=1$. This gives 
$$\mathcal{C}^2 \ge \sum\limits_{i=1}^r m_i^2 -1,$$ which contradicts \eqref{eq:1}.

\noindent
\textbf{Case (ii):} $\widetilde{\mathcal{C}} = \widetilde{\Gamma_e}$. As $\Gamma_e$ is smooth, by \eqref{Xu's like inequality}, $\mathcal{C}^2 \ge \sum\limits_{i=1}^r m_i^2 -1$, which contradicts \eqref{eq:1}. 

So we conclude that 
\begin{eqnarray}\label{eq:3}
\mathcal{C}^2 \ge \sum\limits_{i=1}^r m_i^2 -1.
\end{eqnarray}
Now, suppose that the equality holds in \eqref{eq:3}. That is, $\mathcal{C}^2 = \sum\limits_{i=1}^r m_i^2 -1$. This means that $\widetilde{\mathcal{C}}^2 = -1$. So, by Conjecture \ref{conjecture-1}, $\widetilde{\mathcal{C}}$ could be of the types (i), (ii) or (iii). If $\widetilde{\mathcal{C}}$ is a $(-1)$-curve or $\widetilde{\mathcal{C}} = \widetilde{\Gamma_e}$, we are done with the proof, irrespective of the value of $e$. Suppose that $\widetilde{\mathcal{C}}$ is the strict transform of a curve $\mathcal{C}$ such that $\mathcal{C}^2 \le 0$. In this case, by \eqref{eq:2}, $\mathcal{C}^2 = 0$. Hence $\sum\limits_{i=1}^r m_i^2 =1$. So there exists $i$ such that $m_i = 1$ and $m_j = 0$ for $j\ne i$. 

Now, let $\mathcal{C} = \alpha \Gamma_e + \beta f$. Then $\mathcal{C}^2 = 0$ implies that 
		\begin{eqnarray*}
			&&2 \alpha \beta - \alpha^2 e =0 \\
			&\implies &\alpha (2 \beta - \alpha e) =0  \\
			&\implies& \alpha =0 \; \; \text{or} \;  \; 2 \beta - \alpha e =  0 \\
			&\implies &\alpha =0 \; \; \text{or} \;\;   \beta =  e =  0 \;\; \text{or} \;\; \alpha \neq 0, e \neq 0, 2\beta = \alpha e. 
		\end{eqnarray*}

If $\alpha=0$, then $\mathcal{C} = f$, and in this case $\widetilde{\mathcal{C}}$ is a $(-1)$-curve irrespective of $e$. So we are done. If $\beta = e = 0$, then $\mathcal{C} = \alpha\Gamma_e$, and we are done. Finally, if $\alpha \neq 0$, $\beta \neq 0$ and $2\beta = \alpha e$, then $\mathcal{C} = \alpha \Gamma_e + \frac{1}{2} \alpha e f$. Now, by \cite[Remark, p122]{JeffRosoff2002}, for such a curve $\mathcal{C}$, $\mid \mathcal{C} \mid = \{ \mathcal{C} \}$. So such curves $\mathcal{C}$ form a countable collection. As the points $p_1, \dots , p_r$ are very general, they can be chosen to be outside all those curves $\mathcal{C}$. This means that we do not have to consider such curves. This completes the proof of Conjecture \ref{conjecture-2} assuming Conjecture \ref{conjecture-1}. \\

Conversely, suppose that Conjecture \ref{conjecture-2} is true. Let $C \subset X_r$ be a reduced irreducible curve with $C^2 < 0$. We consider different cases separately. \\

\noindent
\textbf{Case (i):} If $C$ is an exceptional divisor, then it is a $(-1)$-curve. So we are done. \\

\noindent
\textbf{Case (ii):} If $\pi(C) = \mathcal{C}$ and $\mathcal{C}$ does not  pass through any of the points $p_1, \dots , p_r$, then $0 > C^2 = \mathcal{C}^2$. This means that $C = \widetilde{\mathcal{C}}$ with $\mathcal{C}^2 <0$. So we are done. \\

\noindent 
\textbf{Case (iii):} Suppose that $\pi (C) = \mathcal{C}$ passes through at least one $p_i$ and that $m_i = \text{mult}_{p_i}\mathcal{C}$. Then, by Conjecture \ref{conjecture-2},
\[
\mathcal{C}^2 \ge \sum\limits_{i =1}^r m_i^2 -1. 
\]
This implies that
\[
C^2 = \mathcal{C}^2 - \sum\limits_{i =1}^r m_i^2   \ge -1.
\] 
As $C^2 < 0$, this means $C^2 = -1$. Therefore $\mathcal{C}^2 = \sum\limits_{i =1}^r m_i^2 -1$. So, again by Conjecture \ref{conjecture-2}, if $e \neq 0$, $C = \widetilde{\mathcal{C}}$ is a $(-1)$-curve or $\mathcal{C} = \Gamma_e$, and if $e = 0$, $C = \widetilde{\mathcal{C}}$ is a $(-1)$-curve or $\mathcal{C} = \alpha \Gamma_e$, where $\alpha \ge 1$.  In the former case, this implies that $C$ is of types (i) or (iii) as in Conjecture \ref{conjecture-1}. In the latter case, $C$ is of type (ii) if $\mathcal{C} = \alpha \Gamma_e$ where $\alpha > 1$, and of types (i) and (iii) otherwise. This completes the proof.
\end{proof}
\end{proposition}

We now prove that Conjecture \ref{conjecture-1} for any ruled surface implies a statement similar to the Nagata-Biran-Szemberg conjecture, which is a more general statement for any smooth projective surface. We first state the Nagata-Biran-Szemberg Conjecture in its original form (see \cite[Section 4.1]{Sze}).

\begin{conjecture}[Nagata-Biran-Szemberg Conjecture] \label{Nagata Biran Szemberg Conjecture}

Let $Y$ be a smooth projective surface and $q_1, \dots , q_r$ be $r$ very general points on $Y$. Let $\mathcal{L}$ be an ample line bundle on $Y$. 
Then 
\begin{equation}\label{NBS} 
\varepsilon(Y,\mathcal{L},r) = \sqrt{\frac{\mathcal{L}^2}{r}}, \quad \quad \forall \; r \ge k_0^2 \mathcal{L}^2,
\end{equation}
where $\varepsilon(Y,\mathcal{L},r)$ is the multi-point Seshadri constant of $\mathcal{L}$ at $q_1, \dots, q_r$ and $k_0$ is the least positive integer such that $\mid k_0 \mathcal{L} \mid$ has a non-rational non-singular curve.
\end{conjecture}
\begin{remark}
By definition of Seshadri constants, \eqref{NBS} is equivalent to the inequality
\[
\mathcal{L} \cdot \mathcal{C} \ge \left(\sum\limits_{i=1}^r n_i\right) \sqrt{\frac{\mathcal{L}^2}{r}}, \quad \quad \forall \; r \ge k_0^2 \mathcal{L}^2,
\]
where $\mathcal{C}$ is any reduced irreducible curve in $Y$ with $n_i = \text{mult}_{q_i}\mathcal{C}$. 
\end{remark}

\begin{proposition} \label{NBS analog in our case}
    Let $X$ be a ruled surface and let $p_1, \dots, p_r \in X$ be very general points. Let $\mathcal{L} = a \Gamma_e + b f$ be a fixed ample line bundle on $X$. Suppose that Conjecture \ref{conjecture-1} is true. Then the equality \eqref{NBS} holds for sufficiently large $r$. 
    
More precisely, for a sufficiently large $r$, depending only on $\mathcal L$, the following holds: 
     for any reduced irreducible curve  $\mathcal{C} \subset X$, 
     we have
    \begin{equation} \label{NBS inequality}
	\mathcal{L} \cdot \mathcal{C}\geq \left({\sum\limits_{i=1}^r \text{mult}_{p_i}\mathcal{C}} \right)\sqrt{\frac{\mathcal{L}^2}{r}}.
    \end{equation}
\end{proposition}
\begin{proof}
    Suppose that Conjecture \ref{conjecture-1} is true and that there exists a curve $\mathcal{C} = \alpha \Gamma_e + \beta f$ on $X$ passing through some $p_i$ such that for every positive integer $r$
    \begin{equation} \label{violating inequality}
        \mathcal{L} \cdot \mathcal{C}< \left(\sum\limits_{i=1}^r n_i\right)\sqrt{\frac{\mathcal{L}^2}{r}},
    \end{equation} where $n_i = \text{mult}_{p_i}\mathcal{C}$ for $1 \le i \le r$. Since $\mathcal L$ is ample, we must have $n_i > 0$ for some $i$. 
    
 By squaring the inequality \eqref{violating inequality} on both sides, then using Cauchy-Schwarz inequality and the fact that $(\mathcal{L} \cdot \mathcal{C})^2 \geq \mathcal{L}^2 \mathcal{C}^2$ (Hodge index theorem),  we get $\mathcal{C}^2 < \sum\limits_{i=1}^r n_i^2$. Now, if $\widetilde{\mathcal{C}}$ denotes the strict transform of $\mathcal{C}$ in $X_r$, then since $\widetilde{\mathcal{C}}^2 = \mathcal{C}^2 - \sum\limits_{i=1}^r n_i^2$, we can conclude that $\widetilde{\mathcal{C}}^2 < 0$ in $X_r$. So, by Conjecture \ref{conjecture-1}, we have 
 \begin{enumerate}
     \item[(i)] $\widetilde{\mathcal{C}}$ is a $(-1)$-curve, or 
     \item[(ii)]  $\mathcal{C}^2 \leq 0$, or
     \item[(iii)]  $\mathcal{C} = \Gamma_e$.
 \end{enumerate}
 
 We consider each case separately.

 \noindent \textbf{Case (i) :} Suppose that $\widetilde{\mathcal{C}}$ is a $(-1)$-curve. 
 So $-K_{X_r} \cdot \widetilde{\mathcal{C}} =  1$. 
 This implies the following: 
\[   (2 H_e + (e+2-2g)F_e - \sum\limits_{i=1}^r E_i) \cdot (\alpha H_e + \beta F_e - \sum\limits_{i=1}^r n_i E_i)  =  1 \nonumber \]
\begin{eqnarray*}
    \implies 2 \alpha + 2 \beta - \alpha e - 2 \alpha g -1  &=&  \sum\limits_{i=1}^r n_i\nonumber \\ 
    & > & \sqrt{\frac{r}{\mathcal{L}^2}}~(\mathcal{L} \cdot \mathcal{C})~~~\text{by}~\eqref{violating inequality} \nonumber \\
    & = & \sqrt{\frac{r}{\mathcal{L}^2}}~ (a \Gamma_e + b f) \cdot (\alpha \Gamma_e + \beta f) \nonumber \\
    & = & \sqrt{\frac{r}{\mathcal{L}^2}}~(-a \alpha e + \beta a + b \alpha).
 \end{eqnarray*}
 
 From the above inequalities, we conclude that if there is a curve violating \eqref{NBS inequality}, we must have the following:
 \begin{equation}\label{violating inequality for a -1-curve}
     \sqrt{\frac{r}{\mathcal{L}^2}}~(-a \alpha e + \beta a + b \alpha) ~ < ~ 2 \alpha + 2 \beta - \alpha e - 2 \alpha g -1.
 \end{equation}

 We now show that if $r$ is chosen to be sufficiently large, \eqref{violating inequality for a -1-curve} does not hold. We consider 
various sub-cases depending on $e$. \\

 \noindent \textbf{Sub-Case (a):} Suppose that $e > 0$. Since $b > ae$,  the inequality \eqref{violating inequality for a -1-curve} implies 
 \begin{equation} \label{violating inequality for e > 0 case}
     \sqrt{\frac{r}{\mathcal{L}^2}}~ a \beta ~ < ~ 2 \alpha + 2 \beta - \alpha e - 2 \alpha g -1.
 \end{equation}
 Now if $\beta = 0$, then by \cite[Proposition 2.20, Chapter V]{Hartshorne1977}, $\mathcal{C} = \Gamma_e$, which we handle separately in \textbf{Case (iii)}. So, again by \cite[Proposition 2.20, Chapter V]{Hartshorne1977}, we can assume $\beta > 0$. As $a > 0$, multiplying both sides of the inequality \eqref{violating inequality for e > 0 case} by $\frac{1}{a \beta}$, we have 
 \begin{eqnarray*}
     \sqrt{\frac{r}{\mathcal{L}^2}} & < & \frac{2 \alpha}{a \beta} + \frac{2}{a} - \frac{\alpha e}{a \beta} - \frac{2 \alpha g}{a \beta} - \frac{1}{a \beta} \nonumber \\
     & < & 2 + \frac{\alpha}{a \beta} (2 - e - 2g) \nonumber \\
     & \leq & 2 + \frac{1}{ae} (2 - e - 2g) ~~~~ \text{as}~ \beta \geq \alpha e,~\text{and so}~\frac{\alpha}{\beta} \leq \frac{1}{e} \nonumber \\
     & < & 2 + (2 -e - 2g) \nonumber \\
     & = & 4 - e - 2g.
 \end{eqnarray*}
Hence $r < (4 - e - 2g) ^ 2 \mathcal{L}^2$. So we can conclude that for all $r \geq (4 - e - 2g) ^ 2 \mathcal{L}^2$ and for all reduced irreducible curves $\mathcal{C}$, the inequality \eqref{NBS inequality} holds. \\
 
 \noindent \textbf{Sub-Case (b):} Suppose that $e < 0$. As $b > \frac{ae}{2}$, the inequality \eqref{violating inequality for a -1-curve} implies
 \begin{equation} \label{violating inequality for e < 0 case}
     \sqrt{\frac{r}{\mathcal{L}^2}}~(a (\beta - \frac{\alpha e}{2})) ~ < ~ 2 \alpha + (2 \beta - \alpha e) - 2 \alpha g -1.
 \end{equation}
 Again, if $\beta = 0$, then for $\alpha = 1$, we have $\mathcal{C} = \Gamma_e$, which we handle in \textbf{Case (iii)}. If $\beta = 0$ and $\alpha > 1$, then the inequality \eqref{violating inequality for e < 0 case} reduces to
 \begin{eqnarray*}
      \sqrt{\frac{r}{\mathcal{L}^2}} (-a \frac{\alpha e}{2}) & < & 2 \alpha - \alpha e - 2 \alpha g  \\ \implies \sqrt{\frac{r}{\mathcal{L}^2}} & < & \frac{4}{-ae} + 2 \\
      & \le & 6.
 \end{eqnarray*}
 This implies that for all $r \geq 36 \mathcal{L}^2$, we have the desired inequality.
 
 Now suppose that $\beta \neq 0$. If $\alpha = 0$, then $\mathcal{C} = f$. In this case, the inequality \eqref{violating inequality for e < 0 case} implies $\sqrt{\frac{r}{\mathcal{L}^2}} ~ < ~ \frac{1}{a} ~ \le ~ 1$. Therefore, for all $r \geq \mathcal{L}^2$, the inequality \eqref{NBS inequality} is true in this case. 

 Now, if $\alpha = 1$ and $\beta > 0$, \eqref{violating inequality for e < 0 case} gives 
 \begin{eqnarray*}
     \sqrt{\frac{r}{\mathcal{L}^2}}~(a (\beta - \frac{e}{2})) & < &  2  + (2 \beta -  e) - 2g -1 \nonumber \\
     \implies \sqrt{\frac{r}{\mathcal{L}^2}} & < & \frac{(2 - 2g - 1)}{a(\beta - \frac{e}{2})} + \frac{2}{a}.
 \end{eqnarray*}
 As $\beta - \frac{e}{2} \ge 1$, the above inequality gives
 \begin{equation*}
     \sqrt{\frac{r}{\mathcal{L}^2}} ~ < ~ \frac{1}{a} + \frac{2}{a} ~ \leq ~ 3.
 \end{equation*}
 This implies, for all $r \geq 9 \mathcal{L}^2$, the inequality \eqref{NBS inequality} holds in this case.

 If $\alpha \geq 2$ and $\beta > \frac{\alpha e}{2}$, we multiply both sides of the inequality \eqref{violating inequality for e < 0 case} by $\frac{1}{a(\beta - \frac{\alpha e}{2})}$ to get 
 \begin{equation*}
     \sqrt{\frac{r}{\mathcal{L}^2}} ~ < ~ \frac{2}{a} \cdot \frac{2\alpha}{2\beta - \alpha e} + \frac{2}{a}.
 \end{equation*}
 Now, if $\beta \geq 0$ and $e < -1$, we can conclude that $\frac{2\alpha}{2\beta - \alpha e} \leq 1$. Also, if $\beta \geq 0$ and $e = -1$, we can conclude that $\frac{2\alpha}{2\beta - \alpha e} \leq 2$. Therefore we have 
 \begin{equation*}
     \sqrt{\frac{r}{\mathcal{L}^2}} ~ < ~ \frac{4}{a} ~ \leq ~ 4, ~\text{if}~e < -1   
 \end{equation*}
 \text{and}
 \begin{equation*}
     \sqrt{\frac{r}{\mathcal{L}^2}} ~ < ~ \frac{6}{a} ~ \le ~ 6, ~\text{if}~e =-1.
 \end{equation*}
 This implies that if $e < -1$, then for all $r \geq 16 \mathcal{L}^2$, the inequality \eqref{NBS inequality} holds and if $e = -1$, for all $r \geq 36 \mathcal{L}^2$, the inequality \eqref{NBS inequality} holds. 

 Finally, if $\beta < 0$, using the fact that $\beta \geq \frac{\alpha e}{2}$, the inequality \eqref{violating inequality for a -1-curve} implies
 \begin{eqnarray*}
     \sqrt{\frac{r}{\mathcal{L}^2}} (\alpha(b-\frac{ae}{2})) & < & (2-e)\alpha ~~~~
     \nonumber \\
     \implies \sqrt{\frac{r}{\mathcal{L}^2}} & < & \frac{2-e}{b - \frac{ae}{2}} \\
     & \le  & 4-2e, ~~~~ (\text{as}~b-\frac{ae}{2} \ge \frac{1}{2}).
 \end{eqnarray*}
 Therefore, for all $r \geq (4-2e)^2 \mathcal{L}^2$, we get the inequality \eqref{NBS inequality}. \\

 \noindent \textbf{Sub-Case (c):} Suppose that $e =0$. The inequality \eqref{violating inequality for a -1-curve} then implies
 \begin{equation*}
     \sqrt{\frac{r}{\mathcal{L}^2}}~(a \beta + b \alpha) ~ < ~ 2 \alpha + 2 \beta - 2 \alpha g - 1.
 \end{equation*}
 We know $\alpha \geq 0$, $\beta \geq 0$. First, assume both $\alpha$ and $\beta$ are positive.  Multiplying the above inequality by $\frac{1}{a \beta + b \alpha}$ and excluding the negative terms on the right hand side, we have
 \begin{eqnarray*}
     \sqrt{\frac{r}{\mathcal{L}^2}} & < & \frac{2 \alpha}{a \beta + b \alpha} + \frac{2 \beta}{a \beta + b \alpha} \\
     & < & \frac{2 \alpha}{b \alpha} + \frac{2 \beta}{a \beta} \\
     & \le  & 4.
 \end{eqnarray*}
 This implies that if $\alpha\ne 0, \beta \neq 0$, then for all $r \geq 16 \mathcal{L}^2$, the inequality \eqref{NBS inequality} holds. 
 
 If $\alpha = 0$, then $\beta =1$. In this case, the inequality \eqref{violating inequality for a -1-curve} reduces to $\sqrt{\frac{r}{\mathcal{L}^2}} < \frac{1}{a} \le 1$. So, in this case, for all $r \geq \mathcal{L}^2$, we get the desired	 result. And finally, if $\beta = 0$, then 
 \begin{eqnarray*}
     \sqrt{\frac{r}{\mathcal{L}^2}}~(b \alpha) & < & 2 \alpha - 2 \alpha g - 1  <  2 \alpha \\
     \implies \sqrt{\frac{r}{\mathcal{L}^2}} & < & \frac{2}{b} \le  2.
 \end{eqnarray*}
 So, for all $r \geq 4 \mathcal{L}^2$, we get the desired result. \\

 \noindent \textbf{Case (ii):} Suppose that $\mathcal{C}^2 \leq 0$. By \cite{Ein-Lazarsfeld1993}, $\mathcal{C}^2 \geq m(m-1)$, where $m$ is the multiplicity of $\mathcal{C}$ at a very general point. Now since  $p_i$ are very general, it follows that  $\mathcal{C}^2 \ge 0$. So $\mathcal{C}^2 = 0$, which further implies $\alpha(2\beta-\alpha e) = 0$. This means $\alpha = 0$ or $2\beta-\alpha e = 0$. 
 
 If $\alpha = 0$, then $\mathcal{C} = f$, and so $\widetilde{\mathcal{C}} = \widetilde{f}$ is a $(-1)$-curve, which is handled in \textbf{Case (i)}. If $2\beta-\alpha e = 0$, this could mean $\beta = e = 0$, in which case $\mathcal{C} = \alpha\Gamma_e$, or, this could also mean $\beta \neq 0$, but $\beta = \frac{\alpha e}{2}$. 
 
 Now, suppose that $\beta = e = 0$ and $\mathcal{C} = \alpha\Gamma_e$. Then $\mathcal{C}^2 = \alpha^2\Gamma_e ^2 = 0$. So, by \eqref{Xu's like inequality} 
 \begin{equation*}
     \sum\limits_{i=1}^r n_i^2 - n = 0,
 \end{equation*}
  where $n = \text{min} \{n_i~|~n_i \neq 0\}$. This implies that $\sum\limits_{i=1}^r n_i^2 = n$, which further implies that 
  $n_i=1$ for some $i$ and $n_j =0$ for $j \ne i$. 
  As $$\mathcal{L} \cdot \mathcal{C} = \mathcal{L} \cdot \alpha\Gamma_e = \alpha b, \;\; \mathcal{L}^2 = 2ab \text{ and } \sum\limits_{i=1}^r n_i = 1,$$ the inequality \eqref{violating inequality} becomes $\alpha b < \sqrt{\frac{2ab}{r}}$. This is equivalent to $r < \frac{2ab}{\alpha^2 b^2} \le \frac{2a}{b}$. So, if we choose $r \geq \frac{2a}{b}$, the inequality \eqref{NBS inequality} holds.  
 
 Suppose that $\beta \neq 0$ and $\beta = \frac{\alpha e}{2}$. Now, if $e > 0$, then $\beta \geq \alpha e$. So $\beta = \frac{\alpha e}{2}$ can happen only when $e < 0$. But then $\mathcal{C} = \alpha \Gamma_e + (\frac{\alpha e}{2})f$. So, by the arguments as in \textbf{Case (ii)} of Proposition \ref{equivalence of conjectures 1 and 2}, $\mathcal{C}$ cannot pass through a very general point. \\

 \noindent \textbf{Case (iii):} Suppose that $\mathcal{C} = \Gamma_e$. In this case, if $e > 0$, we have $\Gamma_e ^ 2 < 0$. So the linear system $|\mathcal{C}|$ is equal to $\{\mathcal{C}\}$. So $\mathcal{C}$ will not pass through a very general point $p_i$.

 If $e = 0$, then $\mathcal{C}^2 = \Gamma_e ^2 = 0$. So, by \eqref{Xu's like inequality},  
 \begin{equation*}
     \sum\limits_{i=1}^r n_i^2 - n = 0,
 \end{equation*}
  where $n = \text{min} \{n_i~|~n_i \neq 0\}$. This implies that $\sum\limits_{i=1}^r n_i^2 = n$, which further implies that there exists exactly one $i$ such that $n_i \neq 0$. In fact, for that particular $i$, $n_i =1$ as $\Gamma_e$ is smooth. As $\mathcal{L} \cdot \mathcal{C} = \mathcal{L} \cdot \Gamma_e = b$, $\mathcal{L}^2 = 2ab$ and $\sum\limits_{i=1}^r n_i = 1$, the inequality \eqref{violating inequality} becomes $b < \sqrt{\frac{2ab}{r}}$. This is equivalent to $r < \frac{2ab}{b^2} = \frac{2a}{b}$. So if we choose $r \geq \frac{2a}{b}$, the inequality \eqref{NBS inequality} holds. 

  Finally, if $e < 0$, by \eqref{Xu's like inequality}, we have $-e = \mathcal{C}^2 \geq \sum\limits_{i=1}^r n_i^2 - 1$. So  $1-e \geq \sum\limits_{i=1}^r n_i ^2 = \sum\limits_{i=1}^r n_i$ as each $n_i \leq 1$. Therefore $$\frac{\mathcal{L} \cdot \mathcal{C}}{\sum\limits_{i=1}^r n_i} = \frac{b - ae}{\sum\limits_{i=1}^r n_i} \geq \frac{b-ae}{1-e}.$$ If we choose $r$ such that $\frac{b - ae}{1-e} \geq \sqrt{\frac{\mathcal{L}^2}{r}}$, then the desired inequality holds. In other words, for all $r \geq (\frac{1-e}{b- ae})^2 \mathcal{L}^2$, the inequality \eqref{NBS inequality} is true. 
 \end{proof}

 \begin{remark}\label{NBS for sL}
   Suppose that  Conjecture \ref{conjecture-1} is true and let $\mathcal L$ be an ample line bundle on $X$. By Proposition \ref{NBS analog in our case}, 
   there exists $r_0$ such that \eqref{NBS inequality} holds for all $r\ge r_0$. Then for every positive integer $s$ and $r\ge r_0$, the inequality \eqref{NBS inequality}
   also holds for $s\mathcal L$. 
 \end{remark}

\section{Seshadri constants}\label{section3}
In this section, we exhibit irrational Seshadri constants on general blow-ups of ruled surfaces assuming that Conjecture \ref{conjecture-1} is true. 

We start with a preliminary result giving a characterization of ampleness for \textit{uniform} line bundles (see the statement of Theorem \ref{ampleness criterion} below) on general blow-ups of arbitrary surfaces, assuming that Conjecture 
\ref{Nagata Biran Szemberg Conjecture} is true. 
For a smooth projective surface, we know by Nakai-Moishezon criterion that a line bundle $\mathcal{L}$ is ample if and only if $\mathcal{L}^2 > 0$ and $\mathcal{L} \cdot C > 0$ for every reduced irreducible curve $C$. However, the following theorem says that in some specific cases, the first condition mentioned in the above criterion  is itself equivalent to the ampleness of the line bundle.
 \begin{theorem}\label{ampleness criterion}
Let $Y$ be a smooth projective surface and $q_1, \dots , q_r$ be very general points on $Y$. Let $\pi: Y_r \rightarrow Y$ be the blow-up of $Y$ at $q_1, \dots , q_r$. Suppose that $\mathcal{L}$ is an ample line bundle on $Y$ such that  $\frac{\mathcal{L} \cdot \mathcal{C}}{\sum\limits_{i=1}^r n_i} \ge \sqrt{\frac{\mathcal{L}^2}{r}}$, for all reduced irreducible curves $\mathcal{C}$ with $n_i = \text{mult}_{q_i}\mathcal{C}$ and $n_i > 0$ for at least one $i$. Then for $\widetilde{\mathcal{L}} = \pi^*\mathcal{L} - m (E_1 + \cdots + E_r)$ with $m>0$, $\widetilde{\mathcal{L}}$ is ample if and only if ${\widetilde{\mathcal{L}}}^2 >0$.

\begin{proof}
If $\widetilde{\mathcal{L}}$ is ample, then obviously ${\widetilde{\mathcal{L}}}^2 >0$. Conversely, suppose that ${\widetilde{\mathcal{L}}}^2 >0$. It suffices to prove that $\widetilde{\mathcal{L}}\cdot C > 0$ for all reduced irreducible curves $C$ in $Y_r$. Again, we consider case by case.\\

\noindent
\textbf{Case (i):} Suppose that $C= E_i$ for some $i$. Then $\widetilde{\mathcal{L}}\cdot C = m >0$. \\

\noindent
\textbf{Case (ii):} Suppose that $\pi(C)$ does not pass through any $q_i$. Then $\widetilde{\mathcal{L}} \cdot C = \mathcal{L} \cdot \pi(C) >0$, as $\mathcal{L}$ is ample. \\

\noindent
\textbf{Case (iii):} Suppose that $\mathcal C  := \pi(C)$ passes through some $q_i$. 
Then $C = \pi^* \mathcal{C} - \sum\limits_{i= 1}^r n_i E_i$, where $n_i = \text{mult}_{q_i} \mathcal{C}$. Therefore 
\begin{eqnarray}\label{eq:4}
\widetilde{\mathcal{L}} \cdot C = \mathcal{L} \cdot \mathcal{C} - m \sum\limits_{i=1}^r n_i.
\end{eqnarray}
As ${\widetilde{\mathcal{L}}}^2 > 0$, we have $\mathcal{L}^2 - r m^2>0$. Therefore
\begin{eqnarray}\label{eq:5}
\sqrt{\frac{\mathcal{L}^2}{r}} > m.
\end{eqnarray}
By hypothesis, we have $\frac{\mathcal{L} \cdot \mathcal{C}}{\sum\limits_{i=1}^r n_i} \ge \sqrt{\frac{\mathcal{L}^2}{r}}$. This, along with \eqref{eq:5}, implies that 
\begin{eqnarray}\label{eq:6}
\frac{\mathcal{L} \cdot \mathcal{C}}{\sum\limits_{i=1}^r n_i} > m \quad \quad \text{or} \quad \quad \mathcal{L}\cdot \mathcal{C} > m \sum\limits_{i=1}^r n_i.
\end{eqnarray}
Using \eqref{eq:6} in \eqref{eq:4}, we get $\widetilde{\mathcal{L}} \cdot C >0$. So $\widetilde{\mathcal{L}}$ is ample.
\end{proof}
\end{theorem}


As before, let $X \to \Gamma$ be a ruled surface and let $X_r$ be a blow-up of $X$ at $r$ very general points. Recall that $g$ denotes the genus of $\Gamma$. 

Suppose that $L = a H_e + b F_e - \sum\limits_{i=1}^r n_iE_i$ is a line bundle on $X_r$. For $1 \le i \le r$, let
\begin{equation*} 
    H_i' = 2 H_e + (e+2-2g) F_e - \sum\limits_{j=1}^i E_j.
\end{equation*}
Then $L$ can be re-written in the form 
\begin{eqnarray*}
    L &=& (a - 2n_1) H_e + (b- (e+2-2g)n_1) F_e + (n_1 - n_2) H_1' + \dots + \nonumber \\
    & &  (n_{r-1} - n_r) H_{r-1}' + n_r H_r'. 
\end{eqnarray*}
\begin{definition}\label{definition of good form}
    We say a line bundle $L = a H_e + b F_e - \sum\limits_{i=1}^r n_iE_i$ on $X_r$ is in \textit{good form } if the following conditions hold:
    \begin{enumerate}
        \item[(i)] $n_1\geq \dots \ge n_r \geq 0$,
        \item[(ii)] $a \geq 2n_1$,
        \item[(iii)] $b \geq (e+2-2g)n_1$ and
        \item[(iv)] $b \geq ae + n_1$, if $X = \Gamma \times \mathbb{P}^1$ and $b\ge ae$, otherwise.  
    \end{enumerate}
\end{definition}

\begin{proposition}\label{necessary condition for L to be in good form}
    Suppose that $L = a H_e + b F_e - \sum\limits_{i=1}^r n_iE_i$ is a line bundle on $X_r$ such that it is in good form. Then for any curve $C \subset X_r$ such that $C$ is a $(-1)$-curve or $C = \widetilde{\Gamma_e}$, $L\cdot C \geq 0$.
\end{proposition}
\begin{proof}
  Suppose that $C = E_i$ for some $i$. Then $L \cdot C = n_i \geq 0$. So we assume that $C$ is the strict transform of some curve $\mathcal{C} \subset X$. Let $\mathcal{C} = \alpha \Gamma_e + \beta f$, so that $C = \alpha H_e + \beta F_e - \sum\limits_{i=1}^r m_i E_i$, where $m_i = \text{mult}_{p_i}\mathcal{C}$. We again consider both the possibilities for $C$ separately. \\

  \noindent \textbf{Case (i):} Suppose that $C$ is a $(-1)$-curve in $X_r$. In order to show that $L \cdot C \geq 0$, it is sufficient to show that the intersection products of $C$ with $H_e$, $F_e$ and all the $H_i'$ are non-negative (for, since $L$ is in good form, the coefficients of $H_e,F_e$ and all the $H_i'$ are already non-negative, by definition). As each $m_j \geq 0$ and 
  \begin{equation*}
      (H_i' + K_{X_r}) \cdot C = (E_{i+1}+\dots+E_r) \cdot C = m_{i+1}+\dots+m_r \ge 0 , 
  \end{equation*}
  we have $H_i' \cdot C \geq -K_{X_r} \cdot C = 1$, as $C$ is a $(-1)$-curve. Also, $F_e \cdot C = \alpha \geq 0$. Finally, $H_e \cdot C = \beta - \alpha e$. We will now show that $\beta - \alpha e \geq 0$ irrespective of the value of $e$. \\
  
  Suppose that $e < 0$. 
  Then it is easy to see that  $\beta \geq \frac{\alpha e}{2}$. So $\beta - \alpha e \geq - \frac{\alpha e}{2} \geq 0$. \\

  Suppose that $e \geq 0$. If $C \neq \widetilde{\Gamma_e},\widetilde{f}$, then $\beta - \alpha e \geq 0$ as $\mathcal{C}$ is an irreducible curve. If $C = \widetilde{f}$, then $\beta - \alpha e = 1 > 0$. We handle the $C = \widetilde{\Gamma_e}$ in \textbf{Case (ii)}. \\

  \noindent \textbf{Case (ii):} Suppose that $C = \widetilde{\Gamma_e}$. In this case, if $e > 0$, then since $\Gamma_e ^ 2 = -e < 0$, we have $|\Gamma_e| = \{\Gamma_e\}$. Now, since the points $p_i$ are very general, we can choose them to be outside $\Gamma_e$, and therefore, $L \cdot C = b- ae \geq 0$, by condition (iv) in the definition of good form.\\

By \cite[Lemma 17]{Garcia2005},  $h^0(\mathcal{O}_X(\Gamma_e)) = 2$, if $X = \Gamma \times \mathbb{P}^1$ and $h^0(\mathcal{O}_X(\Gamma_e)) = 1$, otherwise.   If $e \leq 0$ and $X \neq \Gamma \times \mathbb{P}^1$, then since $h^0(\mathcal{O}_X (\Gamma_e)) = 1$, we again have $|\Gamma_e| = \{\Gamma_e\}$, which can be handled as above. Finally, if $X = \Gamma \times \mathbb{P}^1$, then since $0 = \Gamma_e^2 \ge \sum\limits_{i=1}^r m_i^2 - m$, we can conclude that exactly one $m_i =1$, all other $m_j$ are zero and $m=1$. This means, $\Gamma_e$ passes through exactly one $p_i$. So, in this case, $L \cdot C = b - ae - n_i \ge 0$ by conditions (i) and (iv) in the definition of a good form.
\end{proof}

Suppose now that $x \in X_r$ is a very general point and $\pi_x : \text{Bl}_x(X_r) \rightarrow X_r$ is the blow-up of $X_r$ at $x$ with $E_x$ as the exceptional divisor. In this setting, we have the following theorem which gives a condition for an ample line bundle to have maximal Seshadri constant at a very general point. 

\begin{theorem} \label{seshadri constant equal to rootLsquare}
Let $X$ be a ruled surface and $X_r$ be the blow-up of $X$ at $r$ very general points. 
Let $x \in X_r$ be a very general point and $\pi_x : {\rm Bl}_x(X_r) \rightarrow X_r$ the blow-up of $X_r$ at $x$ with $E_x$ as the exceptional divisor. Suppose that $L$ is an ample line bundle on $X_r$ such that $\pi_x^{*}L - \sqrt{L^2} E_x$ is in good form. If Conjecture \ref{conjecture-1} is true, then $\varepsilon(X_r, L,x) = \sqrt{L^2}$.
\end{theorem}
\begin{proof}
    Suppose that $\varepsilon = \varepsilon(X_r,L,x) < \sqrt{L^2}$. This implies that there exists a reduced irreducible curve $C \subset X_r$ such that $\varepsilon(X_r,L,x) = \frac{L \cdot C}{\text{mult}_xC }$. So we have
    \begin{equation} \label{inequality contradicting the good form}
        0 ~ = ~ (\pi_x^{*}L - \varepsilon E_x) \cdot \widetilde{C} ~ > ~ (\pi_x^{*}L - \sqrt{L^2} E_x) \cdot \widetilde{C},
    \end{equation}
    where $\widetilde{C}$ is the strict transform of $C$. By the Hodge index theorem, $\widetilde{C}^2 < 0$. So, by Conjecture \ref{conjecture-1}, $\widetilde{C}$ has three possibilities, namely 
    \begin{enumerate}
        \item[(i)] $\widetilde{C}$ is a $(-1)$-curve, 
        \item[(ii)] $\widetilde{C}$ is such that $\mathcal{C}^2 \leq 0$, where $\mathcal{C} = \pi \circ \pi_x (\widetilde{C}) = \pi(C)$, or
        \item[(iii)] $\widetilde{C}= \widetilde{H_e}$, which is the strict transform of $H_e$ on $\text{Bl}_x(X_r)$.
    \end{enumerate}
    If $\widetilde{C}$ is of type (i) or (iii), then since $\pi_x^{*}L - \sqrt{L^2} E_x$ is in good form, Proposition \ref{necessary condition for L to be in good form} will contradict the inequality \eqref{inequality contradicting the good form}, thereby completing the proof. So we are left with proving the result for the case when $\widetilde{C}$ is of type (ii). \\

We claim that $\mathcal{C}^2 \geq 0$. This is clear if $p_i \in \mathcal{C}$ for some $i$, since $p_1,\ldots, p_r$ are very general points. Otherwise, $\mathcal{C}^2 = C^2 \ge 0$ since $C \subset X_r$ passes through a very general point $x \in X_r$. 

    Suppose that $\mathcal{C}^2 \leq 0$. 
        So $\mathcal{C}^2=0$. If $\mathcal{C} = \alpha \Gamma_e + \beta f$, we have $\alpha(2 \beta - \alpha e) = 0$. 
    
    If $\alpha = 0$, then $\mathcal{C} = f$, and so $\widetilde{C}$ will be a $(-1)$-curve, which falls under type (i). 

    Now assume that $\alpha \ne 0$.
    If $\beta = e = 0$, we have $\mathcal{C} = \alpha \Gamma_e$. Then 
    we get the required contradiction since $\alpha \ge 1$ and 
    $(\pi_x^{*}L - \sqrt{L^2} E_x) \cdot \widetilde{H_e} \ge 0$.

    If $\beta \neq 0$ and $\beta = \frac{\alpha e}{2}$, then such curves form a countable collection. As the points that are being blown-up are very general, we can choose them outside this countable collection. This completes the proof.
\end{proof}

Now we show that Conjecture \ref{conjecture-1} guarantees the existence of an ample line bundle satisfying the conditions of Theorem  \ref{seshadri constant equal to rootLsquare}. 
\begin{theorem} \label{construction of the required ample line bundle} Let $X$ be a ruled surface and $X_r$ be the blow-up of $X$ at $r$ very general points. 
Let $x \in X_r$ be a very general point and $\pi_x : {\rm Bl}_x(X_r) \rightarrow X_r$ the blow-up of $X_r$ at $x$ with $E_x$ as the exceptional divisor. Suppose that Conjecture \ref{conjecture-1} is true. Then there exists an ample line bundle $L$ on $X_r$ such that $\pi_x^*L-\sqrt{L^2}E_x$ is in good form.
\end{theorem}
\begin{proof}
   Let $\mathcal L$ be any ample line bundle on $X$.  By Remark \ref{NBS for sL} and Theorem \ref{ampleness criterion} it is clear that for large enough $r$, $\pi^\ast s\mathcal{L}-\sum\limits_{i=1}^r E_i$ is ample if $s^2\mathcal{L}^2>r$. \\
    
    \noindent \textbf{Case (1): } Suppose that $e \geq 0.$ In this case, we know that $\mathcal{L} :=\Gamma_e+(e+1)f$ is an ample line bundle of $X$. So, by Proposition \ref{NBS analog in our case}, for each $r >>0$, the inequality \eqref{NBS inequality} is true. Fix an $r$ such that the inequality \eqref{NBS inequality} is true for all line bundles $s \mathcal{L}$, where $s \in \mathbb{N}$ (this is possible by Remark \ref{NBS for sL}).
    Now consider the line bundle $L=\pi^\ast s\mathcal{L}-\sum\limits_{i=1}^r E_i$ on $X_r$. Our goal is to find a suitable $s$ such that $L$ is ample and $\pi_x^*L-\sqrt{L^2}E_x$ is in good form. \\

Observe that $\mathcal{L}^2 = e+2$. Now, for all $s \in \mathbb{N}$ such that 
\begin{equation} \label{suitable choice of s}
    (s \mathcal{L})^2 = s^2(e+2) > r,
\end{equation}  
$L = \pi^\ast s\mathcal{L}-\sum\limits_{i=1}^r E_i$ is ample. 

        For the line bundle $\pi_x^*L-\sqrt{L^2}E_x$, the conditions (i) to (iv) in the definition of good form translate to 
        \begin{enumerate}
            \item[(i)] $\sqrt{L^2} \geq 1$,
            \item[(ii)] $s \geq 2 \sqrt{L^2}$,
            \item[(iii)] $s e + s \geq (e + 2 - 2g) \sqrt{L^2}$, and
            \item[(iv)] $s \geq \sqrt{L^2}$, if $X = \Gamma \times \mathbb{P}^1$ and $s \ge 0$, otherwise.
        \end{enumerate}
        Clearly, the conditions (i) and (ii) imply (iii) and (iv). Therefore, it is enough to exhibit $L$ satisfying conditions (i) and (ii). By definition of $L$, the condition (i) takes the form
        \begin{eqnarray}\label{condition 1 for good form}
            2s(se+s) - s^2 e - r & \geq & 1 \nonumber \\
            \implies s^2(e + 2) - 1 & \geq & r.
        \end{eqnarray}
        Similarly, condition (ii) takes the form
        \begin{eqnarray}\label{condition 2 for good form}
            s^2 & \geq & 4 (s^2(e+2) -r) \nonumber \\
            \implies 4r & \geq & 4s^2(e + 2) - s^2.
        \end{eqnarray}
        Now, multiplying the inequality \eqref{condition 1 for good form} by 4, we get
        \begin{equation}\label{4 times condition 1 for good form}
            4s^2(e + 2) - 4 \geq  4r.
        \end{equation}
        Combining inequalities \eqref{4 times condition 1 for good form} and \eqref{condition 2 for good form}, we get
        \begin{equation}\label{conditions 1 and 2 for good form combined - case 1}
          4s^2(e + 2) - 4 \geq  4r \geq 4s^2(e + 2) - s^2.  
        \end{equation}
        If $r,s$ satisfy the above inequalities, then they also satisfy \eqref{suitable choice of s}. It is not hard to see that for all $r \ge (15 + 8e)^2 (e+2) - \frac{(15 + 8e)^2}{4}$, we can find an $s$ such that the inequalities \eqref{conditions 1 and 2 for good form combined - case 1} are satisfied. \\

        \noindent \textbf{Case (2):} Suppose that $e < 0$. In this case, we consider $\mathcal{L} = \Gamma_e + f$, which is clearly an ample line bundle on $X$. We then proceed similar to \textbf{Case (1)} and get the inequalities 
        \begin{equation} \label{conditions 1 and 2 for good form combined - case 2}
            4s^2(2-e) - 4 \geq  4r \geq 4s^2(2 - e) - s^2.
        \end{equation}
        Now, similar to \textbf{Case (1)}, for all $r \geq (15 - 8e)^2 (2 - e) - \frac{(15-8e)^2}{4}$, we can find an $s$ such that the inequalities \eqref{conditions 1 and 2 for good form combined - case 2} are satisfied. For such a choice of $r,s$, the line bundle $L = \pi^\ast s\mathcal{L}-\sum\limits_{i=1}^r E_i$ on $X_r$ is ample and $\pi_x^*L-\sqrt{L^2}E_x$ is in good form. 
\end{proof}

Finally, putting together Theorem \ref{seshadri constant equal to rootLsquare} and Theorem \ref{construction of the required ample line bundle}, we prove the main theorem of this note. 
\begin{theorem} \label{irrationality theorem}
Let $X$ be a ruled surface and let $\pi: X_r \to X$ be the blow-up of $X$ at $r$ very general points. 
Suppose that Conjecture \ref{conjecture-1} is true. Then there exists a pair $(r, L)$ where $r$ is a positive integer and $L$ is an ample line bundle on $X_r$ such that $\varepsilon(X_r,L,x)$ is irrational for a very general point $x \in X_r$. 

    In fact, such a pair exists for infinitely many $r$. 
\end{theorem}
\begin{proof}

We first assume $e \ge 0 $. Let $r,s$ be positive integers for which 
the inequalities \eqref{conditions 1 and 2 for good form combined - case 1} are satisfied. Let $\mathcal{L} = \Gamma_e+(e+1)f$ and  $L=\pi^\ast s\mathcal{L}-\sum\limits_{i=1}^r E_i$. 
Then we have 
    \begin{equation*}
        1 \le L^2 = s^2 (e+2) - r \le \frac{s^2}{4}.
    \end{equation*}
    By Theorem \ref{seshadri constant equal to rootLsquare} and Theorem \ref{construction of the required ample line bundle}, $\varepsilon(X_r,L,x) = \sqrt{L^2}$ for a very general point $x \in X_r$. Now, it is not hard to find  $r$ such that $\sqrt{L^2}$ is irrational (in fact, we can find infinitely many such $r$). For example, take $s \ge 3$ and $r = s^2(e+2)-2$. In this case $L^2 = 2$.

    The argument for $e < 0$ is similar. 
\end{proof}

\section*{Acknowledgements}
	
	Authors were partially supported by a grant from Infosys Foundation. The fourth author was supported by the National Board for Higher Mathematics (NBHM), Department of Atomic Energy, Government of India (0204/2/2022/R\&D-II/1785).

\end{document}